\documentclass[10pt]{amsart}

\usepackage{graphicx}

\usepackage{color}

\usepackage[USenglish]{babel}
\usepackage{latexsym}
\usepackage{amsmath}
\usepackage{amsfonts}
\usepackage{amssymb}
\usepackage{esint}
\usepackage[all]{xy}
\renewcommand{\a }{\alpha }
\renewcommand{\b }{\beta }
\renewcommand{\d}{\delta }
\newcommand{\D }{\Delta }

\renewcommand{\l }{\lambda }

\newcommand{\n }{\nabla }

\newcommand{\s }{\sigma }

\renewcommand{\t }{\tau }

\renewcommand{\O }{\Omega }

\newcommand{\ov}{\overline}

\def\p{\partial}

\newcommand{\wtilde }{\widetilde}

\newcommand{\be}{\begin{equation}}
\newcommand{\ee}{\end{equation}}

\newcommand{\R}{\mathbb{R}}

\renewcommand{\P}{\mathbb{P}}

\newcommand{\N}{\mathbb{N}}
\newcommand{\no}{\noindent}
\newcommand{\dis}{\displaystyle}

\newtheorem{theorem}{Theorem}[section]
\newtheorem{proposition}[theorem]{Proposition}

\newtheorem{example}[theorem]{Example}

\newcommand{\bpr}{\begin{proposition}}
\newcommand{\epr}{\end{proposition}}
\newcommand{\bex}{\begin{example}\rm}
\newcommand{\eex}{\end{example}}

\begin{document}

\newtheorem{lem}{Lemma}[section]
\newtheorem{pro}[lem]{Proposition}
\newtheorem{thm}[lem]{Theorem}
\newtheorem{rem}[lem]{Remark}
\newtheorem{cor}[lem]{Corollary}
\newtheorem{df}[lem]{Definition}

\title[Symmetry and uniqueness of Liouville-type problems]{Symmetry and uniqueness of solutions to some Liouville-type  equations and systems}

\author{Changfeng Gui, Aleks Jevnikar, Amir Moradifam}

\address{ Changfeng Gui,~Department of Mathematics, University of Texas at San Antonio, Texas, USA}
\email{changfeng.gui@utsa.edu}

\address{ Aleks Jevnikar,~University of Rome `Tor Vergata', Via della Ricerca Scientifica 1, 00133 Roma, Italy}
\email{jevnikar@mat.uniroma2.it}

\address{Amir Moradifam,~Department of Mathematics, University of California, Riverside, California, USA}
\email{moradifam@math.ucr.edu}

\thanks{The first author is partially supported by  NSF grant DMS-1601885 and NSFC grant No 11371128.
The second author is supported by PRIN12 project: \emph{Variational and Perturbative Aspects of Nonlinear Differential Problems} and FIRB project:
\emph{Analysis and Beyond}.}



\begin{abstract}
We prove symmetry and uniqueness results for three classes of Liouville-type problems arising  in geometry and mathematical physics: asymmetric Sinh-Gordon equation, cosmic string equation and Toda system, under certain assumptions on the mass associated to these problems. The argument is in the spirit of the Sphere Covering Inequality which for the first time is used in treating different exponential nonlinearities and systems.  
\end{abstract}

\maketitle

\section{Introduction}

In this paper, we shall consider three classes of Liouville-type equations and systems: asymmetric Sinh-Gordon equation,  cosmic string equation and Toda system. These problems arise  in geometry and mathematical physics.  We are mainly concerned about  the  symmetry and uniqueness questions under certain assumptions on the mass associated to these problems.  

\medskip
\subsection{Asymmetric Sinh-Gordon equation}
Consider the following version of the asymmetric Sinh-Gordon equation
\begin{equation} \label{eq:sinh-gordon}
\left\{ \begin{array}{rll}
-\D u = & \rho \dfrac{ e^{u}+\frac{\a}{|\a|} e^{\a u} }{\int_\O \left(e^{u}+ e^{\a u}\right) \,dx} & \mbox{in } \O, \vspace{0.2cm}\\
 u = & 0 & \mbox{on } \p \O,
\end{array}
\right.
\end{equation}
where $\a\in[-1,1), \a\neq 0$, $\rho>0$ is a parameter and $\O\subset\R^2$  is a bounded domain with smooth boundary $\p\O$. Equation \eqref{eq:sinh-gordon} is known also as Neri's mean field equation and arises in the context of the statistical mechanics description of $2D$-turbulence introduced in \cite{ons}. In the model where the circulation number density is subject to a probability measure, under a \emph{stochastic} assumption on the vortex intensities one obtains the following equation (see \cite{neri}):
\begin{equation} \label{eq:prob}
\left\{ \begin{array}{rll}
  - \D u = & \rho \dis{\int_{[-1,1]} \b \frac{ e^{\b u} \,\mathcal P(d\b) }{\iint_{[-1,1]\times\O} e^{\b u} \,\mathcal P(d\b) \, dx}}  & \mbox{in } \O, \vspace{0.2cm}\\
       u = & 0 & \mbox{on } \p \O,   
\end{array}
\right.      
\end{equation}
where $u$ stands for the stream function of a turbulent Euler flow, $\mathcal P$ is a Borel probability measure defined in $[-1,1]$ describing the point vortex intensity distribution and $\rho>0$ is a physical constant associated to the inverse temperature.  Equation \eqref{eq:sinh-gordon} is related to the latter model when $\mathcal P$ is supported in two points.

On the other hand, a \emph{deterministic} assumption on the vortex intensities yields the following model (see \cite{sa-su}):
\begin{equation} \label{modif}
\left\{ \begin{array}{rll}
-\D u = & \rho \left(\dfrac{ e^{u}}{\int_\O e^{u} \,dx} + \frac{\a}{|\a|}\dfrac{ e^{\a u} }{\int_\O e^{\a u} \,dx}\right) & \mbox{in } \O, \vspace{0.2cm}\\
 u = & 0 & \mbox{on } \p \O.
\end{array}
\right.
\end{equation}
Concerning the analysis of the latter equation we refer the interested readers to \cite{ao, jev, jev2, jev3, jev4, jwy, jwy2, je-ya, je-ya2, rtzz, ri-ze2}. The arguments presented here do not apply to  \eqref{modif},  and we postpone its analysis to a forthcoming paper.

Observe that by taking $\a=-1$ in \eqref{eq:sinh-gordon} we end up with the standard Sinh-Gordon equation, while for $\mathcal P$ supported in a single point we derive the standard mean field equation
\begin{equation} \label{mf}
\left\{ \begin{array}{rll}
-\D u = & \rho \dfrac{ e^{u}}{\int_\O e^{u} \,dx}  & \mbox{in } \O, \vspace{0.2cm}\\
 u = & 0 & \mbox{on } \p \O,
\end{array}
\right.
\end{equation}
which is related to the prescribed Gaussian curvature problem and Euler flows (see \cite{ba, s} and \cite{clmp, ki}, respectively). The latter equation has been widely studied and we refer to the surveys \cite{mal, tar}. Recently in \cite{gui1, gui2, gui3} the authors proved the Sphere Covering Inequality (see Theorem~\ref{main} below) which leads to several symmetry and uniqueness results for the latter equation. The Sphere Covering Inequality \cite{gui1} will also be a crucial tool in this paper. 

\medskip

Returning to \eqref{eq:sinh-gordon},  some partial existence results and blow-up analysis was carried out in \cite{ri-ze1, ri-ze3}, while a complete existence result for \eqref{eq:prob} with $supp \,\mathcal P \subset [0,1]$ was given in \cite{dm-ri}. On the other hand, we are not aware of any symmetry or uniqueness results for the latter equation with the only exception of \cite{sa-su-ta} where \eqref{modif} is considered. We present here several results in this direction, under natural assumptions both on the parameter $\rho$ and the domain $\O$. Due to different features  of problem \eqref{eq:prob} depending on whether $supp \,\mathcal P \subset [0,1]$ or $supp \,\mathcal P \subset [-1,1]$  we will distinguish  these two cases in the discussion below. In the first situation we may rewrite \eqref{eq:sinh-gordon} as
\begin{equation} \label{eq1}
\left\{ \begin{array}{rll}
-\D u = & \rho \dfrac{ e^{u}+e^{a u} }{\int_\O \left(e^{u}+ e^{a u}\right) \,dx} & \mbox{in } \O, \vspace{0.2cm}\\
 u = & g(x) \geq 0 & \mbox{on } \p \O,
\end{array}
\right.
\end{equation}
with $a\in(0,1)$.  Our first result is the following.

\begin{thm} \label{thm0}
Let $\O \subset \R^2$ be a simply-connected domain and $g\in C(\p\O)$ be a non-negative function.  Suppose $\rho \leq 4\pi$. If $u_1$ and $u_2$ are two solutions of \eqref{eq1} such that
\begin{equation}\label{integralEquality}
\int_\O \left(e^{u_1}+ e^{a u_1}\right) \,dx=\int_\O \left(e^{u_2}+ e^{a u_2}\right) \,dx,
\end{equation}
then $u_1\equiv u_2$. 
\end{thm}

\begin{cor} \label{thm1}
Under the condition of Theorem \ref{thm0}, assume further $\Omega$ and $g$  are evenly symmetric about a line.   Then,  any solution of \eqref{eq1} must be evenly symmetric about that line.  In particular, if $\O$ is radially symmetric and $g$ is a non-negative constant, then $u$ is radially symmetric.
\end{cor}

We will exploit the fact that for $supp \,\mathcal P \subset [0,1]$ equation \eqref{eq:prob} shares some features with the mean field equation \eqref{mf}. Indeed we shall rewrite \eqref{eq:prob} in the form of \eqref{mf} and apply the Sphere Covering Inequality  (see \cite{gui1}) to get the desired results.

\begin{rem}
The argument for Theorem \ref{thm0} can be adapted to treat the more general case where the probability measure $\mathcal P$ in \eqref{eq:prob}  is supported at $(m+1)$ points, i.e.
$$
\left\{ \begin{array}{rll}
-\D u = & \rho \dfrac{ e^{u}+e^{a_1 u}+\cdots+e^{a_m u} }{\int_\O \left(e^{u}+ e^{a_1 u} + \cdots e^{a_m u}\right) \,dx} & \mbox{in } \O, \vspace{0.2cm}\\
 u = & g \geq 0 & \mbox{on } \p \O,
\end{array}
\right.
$$
with $a_i\in(0,1)$ for all $i$. Indeed if  $\rho \leq \dfrac{8\pi}{m+1}$ and
\[\int_\O \left(e^{u_1}+ e^{a_1 u_1} + \cdots e^{a_m u_1}\right) \,dx = \int_\O \left(e^{u_2}+ e^{a_1 u_2} + \cdots e^{a_m u_2}\right) \,dx,\]
then we must necessarily have $u_1\equiv u_2$. In particular, Corollary \ref{thm1} also generalizes to  the above equation. The case where $a_i>1$ fore some $i$ can be carried out as well and we refer to Remark \ref{rem:a} for more details.
\end{rem}

\medskip

On the other hand, for the general case $supp \,\mathcal P \subset [-1,1]$, the problem \eqref{eq:prob} substantially differs from the standard equation \eqref{mf}. In this case we may rewrite \eqref{eq:sinh-gordon} as
\begin{equation} \label{eq2}
\left\{ \begin{array}{rll}
-\D u = & \rho \dfrac{ e^{u}-e^{-a u} }{\int_\O \left(e^{u}+ e^{-a u}\right) \,dx} & \mbox{in } \O, \vspace{0.2cm}\\
 u = & 0 & \mbox{on } \p \O,
\end{array}
\right.
\end{equation}
with $a\in(0,1]$. Observe that $u\equiv 0$ is a solution of the latter problem. We indeed show that for $\rho\leq \dfrac{8\pi}{1+a}$ the trivial solution is the only solution.
\begin{thm} \label{thm2}
Suppose $\rho\leq \dfrac{8\pi}{1+a}$ and $\O\subset \R^2$ simply-connected. Then, equation \eqref{eq2} admits only the trivial solution $u\equiv 0$.
\end{thm}

The proof is based on the Sphere Covering Inequality (see Section \ref{ineq} in \cite{gui1} ). Roughly speaking, letting $v_1=u$, $v_2=-au$ we will consider a symmetrization of $v_2-v_1$ with respect to two suitable measures to get the conclusion.
\begin{rem} \label{rem:a}
Let us point out that in equations \eqref{eq1} and \eqref{eq2} we are considering $a<1$ and $a\leq 1$ (respectively) due to the physical motivations. However, we can treat the case $a>1$ as well. More precisely, letting $v=au$ in \eqref{eq1} we may rewrite the latter equation in a form to which we can apply Theorem \ref{thm1} with a new parameter $\wtilde \rho= a \rho$. Therefore, the conclusions of Theorem \ref{thm0} and Corollary \ref{thm1} still hold true for $\rho\leq \dfrac{4\pi}{a}$ and $a>1$. On the other hand, one can easily see from the proof of Theorem \ref{thm2} that the assumption $a\leq 1$ is not needed and we get the same conclusion for $a>1$. 
\end{rem}

\begin{rem}
The same arguments clearly apply to the following version of \eqref{eq:sinh-gordon}:
\begin{equation} \label{eq:no-int}
\left\{ \begin{array}{rll}
-\D u = &  e^{u}+\frac{\a}{|\a|} e^{\a u}  & \mbox{in } \O, \vspace{0.2cm}\\
 u = & g(x)\geq 0 & \mbox{on } \p \O.
\end{array}
\right.
\end{equation}
We have:
\begin{itemize}
	\item[1.] Let $\a=a\in(0,1)$. Suppose $\O \subset \R^2$ is a simply-connected domain and $g\in C(\p\O)$ is a non-negative function. If $u_1$ and $u_2$ are two solutions of \eqref{eq:no-int} such that
	$$
		\int_\O e^{u_1}\,dx =\int_\O e^{u_2}\,dx \leq 4\pi,
	$$
	then $u_1\equiv u_2$.
	
	\medskip
	
	Moreover, suppose that $\Omega$ and $g$  are evenly symmetric about a line. Let $u$ be a solution of \eqref{eq:no-int} Then, $u$ is evenly symmetric about that line. In particular, if $\O$ is radially symmetric and $g$ is a non-negative constant, then $u$ is radially symmetric.
	
	\bigskip
	
	\item[2.] Let $\a=-a$, $a\in(0,1]$. Suppose $\O\subset \R^2$ simply-connected. If $u$ is a solution of \eqref{eq:no-int} such that
		$$
			\int_\O \left( e^{u}+e^{-au} \right)\,dx \leq \frac{8\pi}{1+a},
		$$
	then $u\equiv 0$.
\end{itemize}

\medskip

\noindent Moreover, similar results hold for $a>1$ (see Remark \ref{rem:a}).

\medskip

The above results follow by suitably adapting the proofs of Theorem \ref{thm0}, Corollary~\ref{thm1} and Theorem \ref{thm2} and we omit the details here.
\end{rem}

\medskip
Finally, we have the following remark concerning the sharpness of the above results.
\begin{rem}
Consider for simplicity the standard Sinh-Gordon equation with $\a=-1$ in \eqref{eq:sinh-gordon}. Even though the associated energy functional is coercive for $\rho<8\pi$ (see \cite{ri-ze3}), we can not extend Theorem \ref{thm0}, Corollary \ref{thm1} and Theorem~\ref{thm2} to the range $\rho\leq 8\pi$ (as it holds for the standard mean field equation \eqref{mf}). In \cite{sa-su-ta} (Section 2) the authors provide non-trivial solutions for \eqref{modif} with $\rho<8\pi$. 
\end{rem}

\

\subsection{Cosmic String Equation}

We will next discuss the following problem to which we will refer to as the cosmic string equation:
\begin{equation} \label{eq:string}
\left\{ \begin{array}{rll}
-\D u = & e^{a u} +h(x)\,e^u & \mbox{in } \O, \vspace{0.2cm}\\
 u = & g(x) \geq 0 & \mbox{on } \p \O,
\end{array}
\right.
\end{equation}  
with $a>0$, $0\in\O\subset\R^2$, and $h$ is of the form
\begin{equation} \label{h}
	h(x)=e^{-4\pi N G_0(x)},
\end{equation} 
where $N\in\N$ and $G_0$ is the Green's function with pole at $0$, i.e.
\begin{equation} \label{green}
\left\{ \begin{array}{rll}
-\D G_0(x) = & \d_0 & \mbox{in } \O, \vspace{0.2cm}\\
 G_0(x) = & 0 & \mbox{on } \p \O.
\end{array}
\right.
\end{equation}
Observe that 
$$
	h>0 \quad \mbox{in } \O\setminus\{0\} \qquad \mbox{and} \qquad h(x)\cong |x|^{2N} \quad \mbox{near } 0.
$$
Equation \eqref{eq:string} describes the behavior of selfgravitating cosmic strings for a massive W-boson model coupled with Einstein's equation where $a$ is a physical parameter and $N$ the string's multiplicity (see \cite{a-o, pt, yang}). Observe that for $a=1$ the equation \eqref{eq:string} is also related to the  Gaussian curvature with conic singularities (see \cite{tar} and references therein). 

Many results concerning \eqref{eq:string} have been established especially for the full plane case. We refer to \cite{chae1, chae2, yang} for existence results, to \cite{pt, pt2} for what concerns symmetry issues, and to \cite{tar2} for blow-up analysis. In particular, in \cite{pt, pt2} the authors provide necessary and  sufficient  conditions for the solvability of \eqref{eq:string} in the full plane in the context of radially symmetric solutions, depending on the values of the total mass $\b=\int_{\R^2}\left(  e^{a u} +|{x}|^{2N}e^u \right)\,dx$. For $N\in(-1,0]$ it follows from a moving plane argument that all the solutions to \eqref{eq:string} are radially symmetric, under suitable assumptions on the domain $\O$. However, it remains an open problem if the results in \cite{pt, pt2} are sharp for the non-radial framework. We prove the following result.

\begin{thm} \label{thm3}
Let $\Omega \subset \R^2$ be a simply-connected domain, $a>0$, $N\geq 0$ and $g\in C(\partial \Omega)$ be non-negative. Suppose $u_1$ and $u_2$ are two distinct solutions of \eqref{eq:string} such that 

\begin{equation}\label{cosmicEnergyBound}
\left\{ \begin{array}{rll}
&\int_\O \left(e^{au_1}+ e^{au_2}\right) \,dx \leq \dfrac{8\pi}{a} \quad & \mbox{if } \quad a\geq 1,\\
&\int_\O \left(e^{u_1}+ e^{u_2}\right) \,dx \leq 8\pi \qquad  & \mbox{if } \quad a<1
\end{array}
\right.
\end{equation}
Then $u_1$ and $u_2$ can not intersect, i.e. either 
\begin{equation}\label{CosmicCanNotIntersect}
u_2>u_1 \qquad \hbox{or} \qquad u_2<u_1 \qquad \hbox{in} \ \ \Omega. 
\end{equation}
\end{thm}

\begin{cor} \label{thm03}
Let $\Omega \subset \R^2$ be a simply-connected domain, $a>0$, $N\geq 0$ and $g\in C(\partial \Omega)$ be non-negative.
 Assume 
\begin{equation}\label{cosmicEnergyBound0}
\gamma:=
\left\{ \begin{array}{rll}
&\int_\O e^{au}\,dx \leq \dfrac{4\pi}{a} \qquad &\mbox{if }  \quad a \geq 1,\\
&\int_\O e^{u}\,dx \leq 4\pi \qquad &\mbox{if } \quad a<1.
\end{array}
\right.
\end{equation}
Then \eqref{eq:string}  has a unique solution $u$ for any $\gamma$ satisfying \eqref{cosmicEnergyBound0}. 
In particular,  if $0 \in \O $ and $\O, g$ are evenly symmetric about a line passing through the origin,  then  $u$ is evenly symmetric about that line.
Consequently,  if $\O$ is radially symmetric about the origin and $g$ is a non-negative constant, then $u$ is radially symmetric about the origin.
\end{cor}

The proof is based on a simple manipulation of equation \eqref{eq:string} and the Sphere Covering Inequality  (see Theorem~\ref{main} below or  \cite{gui1}).

\begin{rem}
 Theorem \ref{thm3} and Corollary \ref{thm03} can be generalized for the following more general equation (we refer to \cite{pt2} for applications of this equation)
$$
\left\{ \begin{array}{rll}
-\D u = & \dis{\sum_{i=0}^m h_i(x)\,e^{a_i u}} & \mbox{in } \O, \vspace{0.2cm}\\
 u = & g(x) \geq 0 & \mbox{on } \p \O,
\end{array}
\right.
$$
where $a_i>0$ and
$$
	h_i(x)= e^{-4\pi N_i G_0(x)}, 
$$
with $N_i\geq 0$ for all $i$. Let $a_M= \max_i\{a_i\}$.  Using similar arguments  as in the proofs of Theorem \ref{thm3},  one can check the assumptions \eqref{cosmicEnergyBound} and \eqref{cosmicEnergyBound0} (where $m=1$) should be replaced by
$$
\int_\O \left(e^{a_M u_1}+ e^{a_M u_2}\right) \,dx \leq \dfrac{16\pi}{a_M (m+1)},
$$
and
$$
\int_\O e^{a_M u}\,dx \leq \dfrac{8\pi}{a_M(m+1)},
$$
respectively. 
\end{rem}

\

\subsection{Liouville-Type Systems}

We also study the following class of Liouville-type systems:
\begin{equation} \label{eq:toda}
\left\{ \begin{array}{ll}
 			\begin{array}{ll}
 			-\D u_1 = & Ae^{u_1} - Be^{u_2} \vspace{0.2cm} \\
 			-\D u_2 = & B'e^{u_2} - A'e^{u_1}
 			\end{array}  & \mbox{in }  \O, \vspace{0.2cm} \\
\ \: u_1=u_2 = g(x) & \mbox{on } \p \O,
\end{array}
\right.
\end{equation} 
with $g\in C(\p\O)$ and 
\begin{equation} \label{cond}
	A, A', B, B' \geq 0, \qquad A+A'=B+B':=M>0.
\end{equation}
Observe that we allow some of the above coefficients to be zero.

The latter system is deeply connected both with geometry and mathematical physics. For example, by taking $A=B'=2$, $B=A'=1$ we recover the $2\times 2$ Toda system which has been extensively studied in the literature. This equation appears in the description of holomorphic curves in $\mathbb{C}\P^N$ (see \cite{bw, ca, lwy}). It also arises in the non-abelian Chern-Simons theory in the context of high critical temperature superconductivity (see \cite{dunne, tar3, yang}). The case $A=B'=1$ and $B=A'=\t$ with a singular source was considered in \cite{pt3} in unbounded domains.

For what concerns Toda-type systems we refer to \cite{jlw, lwyang, lwyz} for blow-up analysis, to \cite{lwy} for classification issues, and to \cite{bjmr, jkm, mr} for existence results. On the other hand, we are not aware of  any symmetry or uniqueness results for Liouville-type systems alike \eqref{eq:toda}. In this direction we provide the following result.

\medskip

\begin{thm} \label{thm4}
Let $(u_1, u_2)$ be a solution of \eqref{eq:toda} and \eqref{cond}. Let $M$ be as defined in \eqref{cond}. Suppose that $\O$ is simply-connected and 
$$
 \int_\O \left( e^{u_1}+e^{u_2} \right)\,dx \leq \frac{8\pi}{M}.
$$ 
Then $u_1\equiv u_2\equiv u$, where $u$ is the unique solution to 
$$
	\left\{ \begin{array}{rll}
-\D u = & D e^{u} & \mbox{in } \O, \vspace{0.2cm}\\
 u = & g(x) & \mbox{on } \p \O,
\end{array}
\right.
$$
and $D:= A-B=B'-A'$. 
\end{thm}

\begin{rem}
For Toda-type systems where $A=B'=2$, $B=A'=1$, the above result asserts that if  $\O$ is simply-connected and 
$$
 \int_\O \left( e^{u_1}+e^{u_2} \right)\,dx \leq \frac{8\pi}{3},
$$ 
 then $u_1\equiv u_2\equiv u$, where $u$ is the unique solution to 
$$
	\left\{ \begin{array}{rll}
-\D u = & e^{u} & \mbox{in } \O, \vspace{0.2cm}\\
 u = & g(x) & \mbox{on } \p \O.
\end{array}
\right.
$$
\end{rem}

Arguing as in the proof of the Sphere Covering Inequality (see Section \ref{ineq} below or \cite{gui1}), we will consider a symmetrization of $u_2-u_1$ with respect to two suitable measures to get the latter result. The uniqueness property will then follow by applying the Sphere Covering inequality to the scalar equation.

\medskip

A similar argument can be carried out for the following singular version of \eqref{eq:toda}:
\begin{equation} \label{eq:toda-sing}
\left\{ \begin{array}{ll}
 			\begin{array}{ll}
 			-\D u_1 = & Ae^{u_1} - Be^{u_2} - 4\pi\a \d_0\vspace{0.2cm} \\
 			-\D u_2 = & B'e^{u_2} - A'e^{u_1} -4\pi \a \d_0
 			\end{array}  & \mbox{in }  \O, \vspace{0.2cm} \\
\ \: u_1=u_2 = g(x) & \mbox{on } \p \O,
\end{array}
\right.
\end{equation}
where $\a\geq 0$ and $0\in\O$. Recall the definitions of $M$, $D$ in \eqref{cond} and in Theorem~ \ref{thm4}, respectively. By using the Green's function $G_0$ with pole at  $0$ as in \eqref{green} we may consider
\begin{equation} \label{wtilde u}
\wtilde u_i(x) = u(x) +4\pi\a G_0(x)
\end{equation}
which satisfies 
$$ 
\left\{ \begin{array}{ll}
 			\begin{array}{ll}
 			-\D \wtilde u_1 = & Ah(x)e^{\wtilde u_1} - Bh(x)e^{\wtilde u_2} \vspace{0.2cm} \\
 			-\D \wtilde u_2 = & B'h(x)e^{\wtilde u_2} - A'h(x)e^{\wtilde u_1} 
 			\end{array}  & \mbox{in }  \O, \vspace{0.2cm} \\
\ \: \wtilde u_1=\wtilde u_2 = g(x) & \mbox{on } \p \O,
\end{array}
\right.
$$
with	$h(x)=e^{-4\pi \a G_0(x)}$. We have the following result.
\begin{thm} \label{thm5}
Let $(u_1, u_2)$ be a solution of \eqref{eq:toda-sing} with $\a\geq 0$ and \eqref{cond}. Let $\wtilde u_i$ be as in \eqref{wtilde u}. Suppose 
 $\O$ is simply-connected and
$$
 \int_\O \left( e^{\wtilde u_1}+e^{\wtilde u_2} \right)\,dx \leq \frac{8\pi}{M}
$$ 
 Then $u_1\equiv u_2\equiv u$, where $u$ is the unique solution to
$$
	\left\{ \begin{array}{rll}
-\D u = & D e^{u} -4\pi\a\d_0 & \mbox{in } \O, \vspace{0.2cm}\\
 u = & g(x) & \mbox{on } \p \O.
\end{array}
\right.
$$
\end{thm}

\medskip

The next remark concerns a possible generalization of the results we have obtained so far for multiply-connected domains.
\begin{rem}
All the  previous results  hold for multiply-connected domains with constant boundary condition, i.e. $g(x)=c \in \R$. This follows from the same arguments and the Sphere Covering Inequality (Theorem~\ref{main}) for multiply-connected domains. See Remark \ref{rem:ineq} below.
\end{rem}

\

The paper is organized as follows. In Section \ref{ineq} we recall the main ingredients of the Sphere Covering Inequality. In Section \ref{sinh-gordon} we present our strategy for proving the uniqueness result of Theorem \ref{thm0}, the symmetry result of Corollary \ref{thm1}, and the uniqueness result of Theorem~\ref{thm2}. In Section \ref{string} we show how to get the no intersection property of Theorem \ref{thm3} and the symmetry property of Corollary \ref{thm03}. In Section~\ref{toda} we provide the proof of the uniqueness result inTheorems \ref{thm4} and \ref{thm5}.

\

\

\begin{center}
\textbf{Notation}
\end{center}

The symbol $B_r(p)$ will denote the open metric ball of radius $r$ and center $p$. Where there is no ambiguity, with a little abuse of notation we will write $x$ and $dx$ to denote $(x,y)\in\R^2$ and the integration with respect to $(x,y)$, respectively. 

\

\section{The Sphere Covering Inequality} \label{ineq}

\medskip

In this section we recall the main ingredients of the Sphere Covering Inequality proved in \cite{gui1} as we will need them in the sequel. Roughly speaking, the latter result asserts that the total area of two distinct surfaces with Gaussian curvature equal to $1$, conformal to the Euclidean unit disk with the same conformal factor on the boundary, must cover the whole unit sphere after a proper rearrangement. See \cite{gui1} for more details. Let us start by recalling the standard Bol's isoperimetric inequality as in \cite{bol}.
\begin{pro} \label{bol}
Let $\O\subset\R^2$ be a simply-connected set and $u\in C^2(\O)$ be such that
$$
	\D u+e^u \geq 0 \qquad \mbox{and} \qquad \int_\O e^u\,dx \leq 8\pi.
$$
Then, for any $\omega\subset\subset\O$ of class $C^1$ it holds
$$
	\left( \int_{\p\omega}e^{\frac{u}{2}}\,d\s \right)^2 \geq \frac 12 \left( \int_\omega e^u\,dx \right)\left( 8\pi-\int_\omega e^u\,dx \right).
$$
\end{pro}

The basic function,   which satisfies  the above properties and  will be used in the sequel,   is the following:
\begin{equation} \label{func}
	U_\l(x)=-2\ln\left( 1+\frac{\l^2|x|^2}{8} \right) + 2\ln \l,
\end{equation}
for $\l>0$. Observe that
$$
	\D U_\l+e^{U_\l} = 0 \qquad \mbox{and} \qquad \int_{B_r(0)} e^{U_\l}\,dx = 8\pi\frac{\l^2 r^2}{8+\l^2 r^2},
$$
for all $r>0$.

Now the idea is  to consider symmetric rearrangements with respect to two distinct measures. More precisely, let $w\in C^2(\ov\O)$ be such that
\begin{equation} \label{est}
	\D w + e^w \geq 0.
\end{equation} 
Then, any function $\phi\in C^2(\ov\O)$ can be equimeasurably rearranged with respect to the measures $e^w \,dx$ and $e^{U_\l}\,dx$ (see \cite{bandle}). Indeed, for $t>\min_{x\in\ov\O}\phi(x)$ let $\mathcal B_t^*$ be the ball centered at the origin such that
$$
	\int_{\mathcal B_t^*} e^{U_\l}\,dx = \int_{\{\phi>t\}} e^w\,dx.
$$
Then, if we let $\phi^*:\mathcal B_t^*\to\R$ to be $\phi^*(x)=\sup\bigr\{t\in\R\,:\,x\in\mathcal B_t^*\bigr\} $, it holds that $\phi^*$ is a symmetric equimeasurable rearrangement of $\phi$ with respect to the measures $e^w \,dx$ and $e^{U_\l}\,dx$, i.e. 
\begin{equation} \label{equi}
	\int_{\{\phi^*>t\}} e^{U_\l}\,dx = \int_{\{\phi>t\}} e^w\,dx,
\end{equation}
for all $t>\min_{x\in\ov\O}\phi(x)$. Moreover, by using the Bol's inequality stated in Proposition \ref{bol} we get the following estimate on the gradient of the rearrangement (see \cite{gui1}).
\begin{pro} \label{rearr}
Let $w\in C^2(\ov\O)$ be such that it satisfies \eqref{est} with $\O\subset\R^2$ being  simply-connected. Let $U_\l$ be as in \eqref{func}. Suppose $\phi\in C^2(\ov\O)$ is such that $\phi \equiv C$ on $\p\O$. If $\phi^*$ is the equimeasurable symmetric rearrangement of $\phi$ with respect to the measures $e^w \,dx$ and $e^{U_\l}\,dx$, then
$$
\int_{\{\phi^*=t\}} |\n \phi^*|\,d\s \leq \int_{\{\phi=t\}} |\n \phi|\,d\s,
$$
\end{pro}
for all $t>\min_{x\in\ov\O}\phi(x)$. 

We shall also need the following counterpart of the Bol's inequality in the radial setting (see \cite{gui1}).
\begin{pro} \label{rad}
Let $\psi\in C^{0,1}(\ov{B_R(0)})$ be a strictly decreasing radial function satisfying 
$$
	\int_{\p B_r(0)} |\n\psi| \,d\s \leq \int_{B_r(0)} e^{\psi}\,dx \quad \mbox{for a.e. } r\in (0,R) \qquad \mbox{and} \qquad \int_{B_R(0)} e^{\psi}\,dx \leq 8\pi.
$$
Then
$$
\left( \int_{\p B_R(0)} e^{\frac{\psi}{2}}\,d\s \right)^2 \geq \frac 12 \left( \int_{B_R(0)} e^{\psi}\,dx \right)\left( 8\pi-\int_{B_R(0)}e^{\psi}\,dx \right).
$$
\end{pro}

The main idea is then to relate strictly decreasing radial function $\psi$ with two radial solutions $U_{\l_1}, U_{\l_2}$ defined in \eqref{func} with $\l_2>\l_1$, such that $\psi=U_{\l_1}=U_{\l_2}$ on $\p B_R(0)$. 
\begin{pro} \label{comp}
$U_{\l_1}, U_{\l_2}$ defined in \eqref{func} with $\l_2>\l_1$. Let $\psi\in C^{0,1}(\ov{B_R(0)})$ be a strictly decreasing radial function satisfying 
\begin{equation} \label{estim}
	\int_{\p B_r(0)} |\n\psi| \,d\s \leq \int_{B_r(0)} e^{\psi}\,dx \quad \mbox{for a.e. } r\in (0,R)
\end{equation}
and $\psi=U_{\l_1}=U_{\l_2}$ on $\p B_R(0)$. Then,  either
$$
	 \int_{B_R(0)} e^{\psi}\,dx \leq \int_{B_R(0)} e^{U_{\l_1}}\,dx \qquad \mbox{or} \qquad \int_{B_R(0)} e^{\psi}\,dx \geq \int_{B_R(0)} e^{U_{\l_2}}\,dx.
$$	
Moreover, we have
$$
\int_{B_R(0)} \left(e^{U_{\l_1}}+e^{U_{\l_2}} \right)\,dx = 8\pi.
$$
\end{pro}

We can now state the Sphere Covering Inequality as in \cite{gui1}.
\begin{thm} \label{main}
Let $\O\subset\R^2$ be a simply-connected set and let $w_i\in C^2(\ov\O)$, $i=1,2$ be such that
\begin{equation} \label{eq}
	\D w_i + e^{w_i}=f_i(x) \qquad \mbox{in } \O,
\end{equation}
where $f_2\geq f_1\geq 0$ in $\O$. Suppose 
$$	
\left\{ \begin{array}{ll}
w_2\geq w_1, \:  w_2 \not\equiv w_1 & \mbox{in } \O, \vspace{0.2cm}\\
w_2=w_1 & \mbox{on } \p \O,
\end{array}
\right.
$$ 
Then, it holds
$$
	\int_\O \left( e^{w_1}+e^{w_2} \right)\,dx \geq 8\pi.
$$
Moreover, if some $f_i\not\equiv 0$ then the latter inequality is strict. 
\end{thm}

The idea is to consider a symmetric rearrangement $\varphi$ of $w_2-w_1$ with respect to the measures $e^{w_1}\,dx$ and $e^{U_{\l_1}}\,dx$ for some suitable $\l_2$. Then, by using equation \eqref{eq} and the properties of the rearrangements (see also Proposition \ref{rearr}), it is possible to show that \eqref{estim} holds true for $\psi=U_{\l_1}+\varphi$. Applying then Proposition \ref{comp} one can deduce that
$$
	\int_\O \left( e^{w_1}+e^{w_2} \right)\,dx \geq \int_{B_R(0)} \left( e^{U_{\l_1}}+e^{U_{\l_2}} \right)\,dx = 8\pi.
$$
See \cite{gui1} for full details.
\begin{rem} \label{rem:ineq}
We point out that the Sphere Covering Inequality holds as long as the Bol's inequality holds. Indeed, if $\Delta w+e^w \geq 0$ in $\Omega$  which is simply-connected, then the Bol's and Sphere Covering Inequalities hold in any region $\Omega_1 \subset \Omega$ for general boundary data. In particular, $\Omega_1$ does not need to be simply-connected.  Moreover, for a multiply-connected domain $\Omega$ the Bol's and Sphere Covering inequalities hold provided we have constant boundary conditions (see \cite{bar-lin}).
\end{rem}

\

\section{Asymmetric Sinh-Gordon equation} \label{sinh-gordon}

\medskip

In this section we study uniqueness and symmetry of solutions of asymmetric Sinh-Gordon equation \eqref{eq:sinh-gordon}, and  prove Theorem \ref{thm0} and Theorem \ref{thm2}. The first one relies mainly on the Sphere Covering Inequality (see Theorem \ref{main}). On the other hand, the second one is based on the arguments which yield the Sphere Covering Inequality, which we collected in Section \ref{ineq}. 

Let us start with the case $supp\, \mathcal P \subset [0,1]$ which we recall here for convenience
\begin{equation} \label{eq1'}
\left\{ \begin{array}{rll}
-\D u = & \rho \dfrac{ e^{u}+e^{a u} }{\int_\O \left(e^{u}+ e^{a u}\right) \,dx} & \mbox{in } \O, \vspace{0.2cm}\\
 u = & g(x) \geq 0 & \mbox{on } \p \O,
\end{array}
\right.
\end{equation}
with $a\in(0,1)$, $\rho>0$, and $g \in C(\partial \Omega)$. 
\medskip

\begin{proof}[Proof of Theorem \ref{thm0}]
Let $u_1$ and $u_2$ be solutions of equation \eqref{eq1'} satisfying the assumptions of Theorem \eqref{thm0}. We aim to show that $u_1\equiv u_2$. 
We proceed by contradiction by assuming that this is not the case. Rewrite  equation \eqref{eq1'} as
$$
\D u +  \rho \dfrac{ 2e^{u} }{\int_\O \left(e^{u}+ e^{a u}\right) \,dx} = \rho \dfrac{ e^{u}-e^{a u} }{\int_\O \left(e^{u}+ e^{a u}\right) \,dx}\,.
$$
Let 
\begin{equation} \label{v}
	v= u+\log 2+\log \rho -\log \left(\int_\O \left(e^{u}+ e^{a u}\right) \,dx\right).
\end{equation}
Then $v$ satisfies 
\begin{equation} \label{equa}
	\D  v + e^{v} = f(u):= \rho \dfrac{ e^{u}-e^{a u} }{\int_\O \left(e^{u}+ e^{a u}\right) \,dx}\,.
\end{equation}
It follows from \eqref{integralEquality} that there exists two regions $\Omega_1, \Omega_2 \subset \Omega$ (not necessarily simply-connected) such that $u_1>u_2$ in $\Omega_1$, $u_2>u_1$ in $\Omega_2$, and $u_1=u_2$ on $\partial \Omega_1 \cup \partial \Omega_2$. We have that $v_1, v_2$ defined by \eqref{v} satisfy 
\[\D  v_i + e^{v_i} = f(u_i) \ \ \hbox{ in}\ \ \Omega.\]
Moreover 
$$
v_1>v_2 \quad \mbox{in }\Omega_1, \quad v_2>v_1 \quad \mbox{in } \Omega_2 \qquad \mbox{and} \qquad v_1=v_2 \quad \mbox{on } \partial \Omega_1 \cup \partial \Omega_2.
$$ 
Since $g \geq 0$, both solutions $u_1$ and $u_2$ are positive in $\Omega$ by the maximum principle. By the latter fact it is also easy to see that 
$$
f(u_1)>f(u_2)>0 \quad \mbox{in } \Omega_1 \qquad \mbox{and} \qquad f(u_2)>f(u_1)>0 \quad \mbox{in } \Omega_2.
$$ 
Therefore, by applying the Sphere Covering Inequality (Theorem \ref{main}, see also Remark \ref{rem:ineq}), we get (observe that $f_i \not\equiv 0$)
$$
	\int_{\Omega} \left( e^{v_1}+e^{v_2} \right)\,dx \geq \int_{\Omega_1} \left( e^{v_1}+e^{v_2} \right)\,dx+\int_{\Omega_2} \left( e^{v_1}+e^{v_2} \right)\,dx>16\pi.
$$
Recalling now the definition of $v$ in \eqref{v} and \eqref{integralEquality} we have
\begin{eqnarray*}
4\rho&=&\dfrac{ 2\rho }{\int_\O \left(e^{u_1}+ e^{a u_1}\right) \,dx} \left(  \int_\O \left(e^{u_1}+ e^{a u_1}\right) \,dx+ \int_\O \left(e^{u_1}+ e^{a u_1}\right) \,dx  \right) \\
&\geq& \dfrac{ 2\rho }{\int_\O \left(e^{u_1}+ e^{a u_1}\right) \,dx} \int_{\O} \left( e^{u_1}+e^{u_2} \right)\,dx=\int_{\Omega} \left( e^{v_1}+e^{v_2} \right)\,dx > 16\pi.
\end{eqnarray*}
Hence $\rho > 4\pi$, which is a contradiction. The proof is now complete. 
\end{proof}

\medskip

\begin{proof}[Proof of Corollary \ref{thm1}]
Without loss of generality we can assume that $\O$ and $g$ are evenly symmetric with respect to the line $y=0$. Suppose $u$ is a solution of \eqref{eq1}, which is not evenly symmetric about $y=0$. Then $u_1=u$ and $u_2(x,y)=u(x,-y)$ are two distinct solutions of \eqref{eq1} satisfying the condition \eqref{integralEquality}. Thus it follows from Theorem~\ref{thm0} that $\rho > 4\pi$. 
\end{proof}

\medskip

We consider now the general case $supp\, \mathcal P \subset [-1,1]$ which yields to \eqref{eq2}, i.e.:
\begin{equation} \label{eq2'}
\left\{ \begin{array}{rll}
-\D u = & \rho \dfrac{ e^{u}-e^{-a u} }{\int_\O \left(e^{u}+ e^{-a u}\right) \,dx} & \mbox{in } \O, \vspace{0.2cm}\\
 u = & 0 & \mbox{on } \p \O,
\end{array}
\right.
\end{equation}
with $a\in(0,1)$, $\rho>0$. We give here the proof of the uniqueness result for the trivial solution $u\equiv0$.

\medskip

\begin{proof}[Proof of Theorem \ref{thm2}]
Let $u$ be a solution of \eqref{eq2'}. We will show that $u\equiv 0$ in $\O$. Assume by contradiction this is not the case and let 
\begin{align} \label{vi}
\begin{split}
	& v_1 = -au + \log \rho -\log \left(\int_\O \left(e^{u}+ e^{-a u}\right) \,dx\right), \\
	& v_2=u+\log \rho -\log \left(\int_\O \left(e^{u}+ e^{-a u}\right) \,dx\right).
\end{split}	
\end{align}
Then we have
$$
	\D(v_2-v_1) +(1+a)\left( e^{v_2}-e^{v_1} \right)=0. 
$$
Letting further
\begin{equation} \label{w}
 w_i = v_i +\log(1+a), \qquad i=1,2,
\end{equation}
we deduce
\begin{equation} \label{equat}
	\D(w_2-w_1) +\left( e^{w_2}-e^{w_1} \right)=0.
\end{equation}
Since $u=0$ on $\p\O$, we get
\begin{equation} \label{prop'}
	w_1=w_2= \log(1+a)+\log \rho -\log \left(\int_\O \left(e^{u}+ e^{-a u}\right) \,dx\right) \qquad \mbox{on } \p\O.
\end{equation}
It follows that  there exists at least one region $\wtilde\O\subseteq\O$ (not necessarily simply-connected) such that 
\begin{equation} \label{set'}
\left\{ \begin{array}{ll}
w_1\neq w_2  & \mbox{in } \wtilde\O, \vspace{0.2cm}\\
w_1 =  w_2 & \mbox{on } \p \wtilde\O,
\end{array}
\right.
\end{equation}
and
\begin{equation} \label{eqq}
	\D(w_2-w_1) +\left( e^{w_2}-e^{w_1} \right)=0 \qquad \mbox{in } \wtilde\O.
\end{equation}
We point out that $\wtilde\O$ may coincide with $\O$. Without loss of generality we may assume $w_2>w_1$. From equation \eqref{eq2'} and the definitions of $w_i$ in \eqref{vi} and  \eqref{w} we derive that
$$
	\D v_1 +ae^{v_1} = ae^{v_2}
$$
and thus
\begin{equation} \label{pos}
	\D w_1 +e^{w_1} = \left(\frac{1}{1+a}\,e^{w_1}+ae^{v_2}\right)>0 \qquad \mbox{in } \O.
\end{equation}
We now proceed as in the proof of the Sphere Covering Inequality. Let $\l_2>\l_1$ be such that $U_{\l_2}>U_{\l_1}$ in $B_1(0)$ and $U_{\l_1}=U_{\l_2}$ on $\p B_1(0)$, where $U_\l$ is given as in \eqref{func}, and such that
$$
	\int_{\wtilde\O} e^{w_1}\,dx = \int_{B_1(0)} e^{U_{\l_1}}\,dx.
$$ 
Since $w_1$ satisfies \eqref{pos} we can find a symmetric equimeasurable rearrangement $\varphi^*$  of $w_2-w_1$ with respect to the two measures $e^{w_1} \,dx$ and $e^{U_{\l_1}}\,dx$. See the discussion after \eqref{est}. In particular we have 
$$
	\int_{\{\varphi^*>t\}} e^{U_{\l_1}}\,dx = \int_{\{w_2-w_1>t\}} e^{w_1}\,dx
$$
for $t \ge 0$.
We first estimate the gradient of the rearrangement by Proposition \ref{rearr}, then exploit equation \eqref{eqq},  the equation satisfied by $U_{\l_1}$ and the properties of the rearrangements  to obtain
\begin{align*}
	\int_{\{ \varphi^*=t \}} |\n\varphi^*|\,d\s & \leq \int_{\{ w_2-w_1=t \}} |\n(w_2-w_1)|\,d\s \\
	& = \int_{\{w_2-w_1>t\}} \left( e^{w_2}-e^{w_1} \right)\,dx \\
		& = \int_{\{ \varphi^*>t \}} e^{U_{\l_1}+\varphi^*}\,dx - \int_{\{ \varphi^*>t \}} e^{U_{\l_1}}\,dx \\
		& = \int_{\{ \varphi^*>t \}} e^{U_{\l_1}+\varphi^*}\,dx - \int_{\{ \varphi^*=t \}} |\n{U_{\l_1}}|\,d\s,
\end{align*}
for a.e. $t>0$. Therefore
$$
	\int_{\{ \varphi^*=t \}} |\n\bigr(U_{\l_1}+\varphi^*\bigr)|\,d\s \leq \int_{\{ \varphi^*>t \}} e^{U_{\l_1}+\varphi^*}\,dx,
$$
for a.e. $t>0$. Since $\varphi^*$ is decreasing by construction, $U_{\l_1}+\varphi^*$ is a strictly decreasing function. Moreover, by the above estimate
we derive
\begin{equation} \label{estimate}
	\int_{\p B_r(0)} |\n\bigr(U_{\l_1}+\varphi^*\bigr)|\,d\s \leq \int_{B_r(0)} e^{U_{\l_1}+\varphi^*}\,dx \qquad \mbox{for a.e. } r>0.
\end{equation}
Furthermore, since $\varphi*\geq 0$, we clearly have
$$
	\int_{B_1(0)} e^{U_{\l_1}+\varphi^*}\,dx \geq \int_{B_1(0)} e^{U_{\l_1}}\,dx.
$$
By the latter estimate, \eqref{pos} and  \eqref{estimate} we can exploit Proposition \ref{comp} with $\psi=U_{\l_1}+\varphi^*$ to get
$$
	\int_{B_1(0)} e^{U_{\l_1}+\varphi^*}\,dx \geq \int_{B_1(0)} e^{U_{\l_2}}\,dx.
$$
Thus 
$$
	\int_{\wtilde\O}\left( e^{w_1}+e^{w_2} \right)\,dx =	\int_{B_1(0)} \left( e^{U_{\l_1}}+e^{U_{\l_1}+\varphi^*} \right)\,dx \geq \int_{B_1(0)} \left( e^{U_{\l_1}}+e^{U_{\l_2}} \right)\,dx = 8\pi.
$$

\medskip

\no Recall now the definitions of $w_i$ in \eqref{vi} and \eqref{w}. We have
$$
 \dfrac{\rho(1+a) }{\int_\O \left(e^{u}+ e^{-a u}\right) \,dx} \,\int_{\wtilde\O}\left( e^{u}+e^{-a u} \right)\,dx \geq 8\pi,
$$
and hence
$$
	\frac{8\pi}{1+a} \leq \dfrac{\rho }{\int_\O \left(e^{u}+ e^{-a u}\right) \,dx} \,\int_{\wtilde\O}\left( e^{u}+e^{-a u} \right)\,dx  \leq \rho.
$$
The above inequality is indeed strict. To see this, we note that the equality would yield the equality in \eqref{estimate} which corresponds to equality in the Bol's inequality in Proposition \ref{bol} for $w_1$ and consequently $w_1$ should satisfy $\D w_1+e^{w_1}=0$,  which contradicts \eqref{pos}. 
In view of   the assumption  $\rho\leq\dfrac{8\pi}{1+a}$,  we therefore   have shown  $u\equiv 0$ in $\O$ as desired.
\end{proof}

\

\section{Cosmic string equation} \label{string}

\medskip

In this section we study the cosmic string equation
\begin{equation} \label{eq:string'}
\left\{ \begin{array}{rll}
-\D u = & e^{a u} +h(x)\,e^u & \mbox{in } \O, \vspace{0.2cm}\\
 u = & g(x) \geq 0 & \mbox{on } \p \O,\end{array}
\right.
\end{equation}  
with $a>0$ and $h$ as in \eqref{h}. We will rewrite this equation and apply the Sphere Covering Inequality, Theorem~\ref{main}, to prove Theorem \ref{thm3}. 

\medskip

\begin{proof}[Proof of Theorem \ref{thm3}]
First suppose $a>1$. Let $u_1$ and $u_2$ be two solutions of \eqref{eq:string'} with $a>1$, $N\geq0$ satisfying \eqref{CosmicCanNotIntersect}. We proceed by contradiction. Suppose there exists $\Omega_1, \Omega_2 \subset \Omega$ (not necessarily simply-connected) such that 
\[u_1>u_2 \ \ \hbox{in} \ \ \Omega_1 \ \ \ \  \ \ \hbox{and}\  \  \ \ \ \  u_2>u_1 \ \ \hbox{in} \ \ \Omega_2 .\]
The equation \eqref{eq:string'} can be rewritten as 
$$
	\D u + 2e^{a u} =  e^{a u} -h(x)\,e^u.
$$
Multiply this equation by $a$ and let 
\begin{equation} \label{v'}
	v= au + \log\left( 2a  \right).
\end{equation}
Then $v$ satisfies
\begin{equation} \label{equa'}
	\D  v + e^{v} = f(u):= a	\bigr(e^{a u} -h(x)\,e^u\bigr).
\end{equation}
Let $v_1, v_2$ be defined by \eqref{v'} ($u$ replaced by $u_1$ and $u_2$, respectively). Then we have 
\[\D  v_i + e^{v_i} = f(u_i) \ \ \hbox{ in}\ \ \Omega.\]
Furthermore, we get
$$
v_1>v_2 \quad \mbox{in }\Omega_1, \quad v_2>v_1 \quad \mbox{in } \Omega_2 \qquad \mbox{and} \qquad v_1=v_2 \quad \mbox{on } \partial \Omega_1 \cup \partial \Omega_2.
$$ 
Since $g \geq 0$, it follows from the maximum principle that both solutions $u_1$ and $u_2$ are positive inside $\Omega$. Note also that $h(x)\leq 1$. It is now easy to see that
$$
f(u_1)>f(u_2)>0 \quad \mbox{in } \Omega_1 \qquad \mbox{and} \qquad f(u_2)>f(u_1)>0 \quad \mbox{in } \Omega_2.
$$
By the Sphere Covering Inequality (Theorem \ref{main}, see also Remark \ref{rem:ineq}) we conclude that
\begin{eqnarray*}
\int_{\Omega} \left(e^{v_1}+e^{v_2}\right)\, dx\geq \int_{\Omega_1}\left( e^{v_1}+e^{v_2}\right)\, dx+\int_{\Omega_2}\left( e^{v_1}+e^{v_2}\right)\, dx> 16\pi. 
\end{eqnarray*}
Using the expression of $v$ in \eqref{v'} we deduce
$$
	 2a  \int_{\O} \left( e^{au_1}+e^{a u_2} \right)\,dx > 16\pi,
$$
which contradicts the assumption 
$$
\int_{\O} \left(e^{a u_1}+e^{a u_2} \right)\,dx \leq \dfrac{8\pi}{a}\,.
$$ 
For what concerns the case $a<1$ we write \eqref{eq:string'} in the form
$$
	\D u + 2e^{u} =  \left(e^u-e^{a u}\right) +\left(e^u-h(x)\,e^u\right).
$$
The argument is then developed as before so we skip the details. The proof is now complete. 
\end{proof}

\begin{proof}[Proof of Corollary \ref{thm03}]
Without loss of generality that $\O$ and $g$ are evenly symmetric with respect to the line $y=0$. Observe that the associated Green's function (and hence $h$, see \eqref{h}) is evenly symmetric with respect to the line $y=0$. We consider just the case $a>1$ since for $a<1$ one can proceed in the same way. Suppose $u$ is a solution of \eqref{eq1} satisfying \eqref{cosmicEnergyBound0}, which is not evenly symmetric about $y=0$. Then $u_1=u$ and $u_2(x,y)=u(x,-y)$ are two distinct intersecting solutions of \eqref{eq:string}. It follows from Theorem \ref{thm3} that 
$$
2\int_{\O} e^{a u} \,dx = \int_{\O} \left(e^{a u_1}+e^{a u_2} \right)\,dx> \dfrac{8\pi}{a}\,.
$$ 
which contradicts \eqref{cosmicEnergyBound0}. 
\end{proof}

\section{Liouville-type systems in domains} \label{toda}

\medskip

In this section we consider the class of Liouville-type systems
\begin{equation} \label{eq:toda'}
\left\{ \begin{array}{ll}
 			\begin{array}{ll}
 			-\D u_1 = & Ae^{u_1} - Be^{u_2} \vspace{0.2cm} \\
 			-\D u_2 = & B'e^{u_2} - A'e^{u_1}
 			\end{array}  & \mbox{in }  \O, \vspace{0.2cm} \\
\ \: u_1=u_2 = g(x) & \mbox{on } \p \O,
\end{array}
\right.
\end{equation} 
where $A, A', B, B'$ satisfy  condition \eqref{cond}, and prove Theorem \ref{thm4}. 

\medskip

\begin{proof}[Proof of Theorem \ref{thm4}]

Let $(u_1, u_2)$ be a solution of \eqref{eq:toda'}. We will prove that there exists a unique $u$ solving a mean field equation as stated in Theorem \ref{thm4} such that $u_1\equiv u_2 \equiv u$ in $\O$. Assume by contradiction $u_1\not\equiv u_2$. As in the proof of Theorem \ref{thm2}, the strategy is to apply the argument of the Sphere Covering Inequality in Theorem \ref{main} (see Section \ref{ineq}) to the functions $u_1$ and $u_2$. We start by recalling that the coefficients in \eqref{eq:toda'} are such that $A+A'=B+B':=M$. Hence 
$$
	\D(u_2-u_1) +M\left( e^{u_2}-e^{u_1} \right)=0. 
$$
Letting 
\begin{equation} \label{w'}
 w_i = u_i +\log M, \qquad i=1,2,
\end{equation}
we deduce that
\begin{equation} \label{equat'}
	\D(w_2-w_1) +\left( e^{w_2}-e^{w_1} \right)=0,
\end{equation}
and 
\begin{equation} \label{prop''}
	w_1=w_2= \log M+g(x) \qquad \mbox{on } \p\O.
\end{equation}
It follows that there exists at least one region $\wtilde\O\subseteq\O$ (not necessarily simply-connected) such that 
\begin{equation} \label{set''}
\left\{ \begin{array}{ll}
w_1\neq w_2  & \mbox{in } \wtilde\O, \vspace{0.2cm}\\
w_1 =  w_2 & \mbox{on } \p \wtilde\O,
\end{array}
\right.
\end{equation}
and
\begin{equation} \label{eqq'}
	\D(w_2-w_1) +\left( e^{w_2}-e^{w_1} \right)=0 \qquad \mbox{in } \O.
\end{equation}
Without loss of generality we can  assume $w_2>w_1$ in $\wtilde \O$. 

Using the first equation in \eqref{eq:toda'}, the definitions of $w_i$ in \eqref{w}, and the fact that $M=A+A'$ we get
$$
	\D u_1 +Ae^{u_1} = Be^{u_2}
$$
and hence
\begin{equation} \label{pos'}
	\D w_1 +e^{w_1} = \left(\frac{A'}{A+A'}\,e^{w_1}+Be^{u_2}\right)\geq 0 \qquad \mbox{in } \O.
\end{equation}
In the above two steps we assumed $A>0$. However, the above holds true even if $A=0$, by simple manipulations. The rest of the argument is very similar to the proof of Theorem \ref{thm2} so we will skip the details. Let $\l_2>\l_1$ be such that $U_{\l_2}>U_{\l_1}$ in $B_1(0)$ and $U_{\l_1}=U_{\l_2}$ on $\p B_1(0)$, where $U_\l$ is given as in \eqref{func}, and 
$$
	\int_{\wtilde\O} e^{w_1}\,dx = \int_{B_1(0)} e^{U_{\l_1}}\,dx.
$$ 
Recalling \eqref{pos'} we can find a symmetric equimeasurable rearrangement $\varphi^*$  of $w_2-w_1$ with respect to the two measures $e^{w_1} \,dx$ and $e^{U_{\l_1}}\,dx$. Reasoning as in the proof of Theorem \ref{thm2} we get  
$$
	\int_{\p B_r(0)} |\n\bigr(U_{\l_1}+\varphi^*\bigr)|\,d\s \leq \int_{B_r(0)} e^{U_{\l_1}+\varphi^*}\,dx \qquad \mbox{for a.e. } r>0.
$$
Furthermore $U_{\l_1}+\varphi^*$ is a strictly decreasing function. Hence from Proposition~\ref{comp} to $\psi=U_{\l_1}+\varphi^*$ we deduce
$$
	\int_{B_1(0)} e^{U_{\l_1}+\varphi^*}\,dx \geq \int_{B_1(0)} e^{U_{\l_2}}\,dx.
$$
Therefore
$$
	\int_{\wtilde\O}\left( e^{w_1}+e^{w_2} \right)\,dx =	\int_{B_1(0)} \left( e^{U_{\l_1}}+e^{U_{\l_1}+\varphi^*} \right)\,dx \geq \int_{B_1(0)} \left( e^{U_{\l_1}}+e^{U_{\l_2}} \right)\,dx = 8\pi.
$$

\medskip

\no It follows from the definitions of $w_i$ that 
$$
	M\,\int_{\wtilde\O}\left( e^{u_1}+e^{u_2} \right)\,dx \geq 8\pi.
$$
Thus 
$$
	\frac{8\pi}{M} \leq \int_{\wtilde\O}\left( e^{u_1}+e^{u_2} \right)\,dx \leq \int_{\O}\left( e^{u_1}+e^{u_2} \right)\,dx .
$$
Arguing as in the proof of Theorem \ref{thm2} it is easy to show that the latter inequality is strict, which is a contradiction. Hence $u_1\equiv u_2$ in $\O$. Letting $u:=u_1=u_2$ and using the system \eqref{eq:toda'} we get
$$
	\left\{ \begin{array}{rll}
-\D u = & D e^{u} & \mbox{in } \O, \vspace{0.2cm}\\
 u = & g(x) & \mbox{on } \p \O,
\end{array}
\right.
$$
where we recall $D:= A-B=A'-B'$. Note that $M:=A+A'=B+B'$ and hence 
$$
	 \int_\O D e^u\,dx \leq  4\pi \frac{D}{M} = 4\pi\frac{A-B}{A+A'} < 4\pi.
$$
Since $\O$ is simply-connected and the latter bound holds true, by the Sphere Covering Inequality of Theorem \ref{main} we deduce that $u$ is unique. This concludes the proof of Theorem \ref{thm4}.
\end{proof}

\medskip

We conclude this section by giving the proof of Theorem~\ref{thm5} regarding the uniqueness of solutions of the system 
\begin{equation} \label{eq:toda-sing'}
\left\{ \begin{array}{ll}
 			\begin{array}{ll}
 			-\D u_1 = & Ae^{u_1} - Be^{u_2} - 4\pi\a \d_0\vspace{0.2cm} \\
 			-\D u_2 = & B'e^{u_2} - A'e^{u_1} -4\pi \a \d_0
 			\end{array}  & \mbox{in }  \O, \vspace{0.2cm} \\
\ \: u_1=u_2 = g(x) & \mbox{on } \p \O. 
\end{array}
\right.
\end{equation}

\medskip

\begin{proof}[Proof of Theorem \ref{thm5}]
Let $(u_1, u_2)$ be a solution of \eqref{eq:toda-sing'} with $\a\geq 0$. By using the Green's function $G_0$ with pole in $0$ as in \eqref{green} we desingularize the problem by setting
$$
\wtilde u_i(x) = u(x) +4\pi\a G_0(x).
$$ 
Indeed \eqref{eq:toda-sing'} is equivalent to 
\begin{equation} \label{eq:toda-sing''}
\left\{ \begin{array}{ll}
 			\begin{array}{ll}
 			-\D \wtilde u_1 = & Ah(x)e^{\wtilde u_1} - Bh(x)e^{\wtilde u_2} \vspace{0.2cm} \\
 			-\D \wtilde u_2 = & B'h(x)e^{\wtilde u_2} - A'h(x)e^{\wtilde u_1} 
 			\end{array}  & \mbox{in }  \O, \vspace{0.2cm} \\
\ \: \wtilde u_1=\wtilde u_2 = g(x) & \mbox{on } \p \O,
\end{array}
\right.
\end{equation}
where 
\begin{equation} \label{h'}
	h(x)=e^{-4\pi \a G_0(x)}.
\end{equation}
Observe that 
$$
	h>0 \quad \mbox{in } \O\setminus\{0\} \qquad \mbox{and} \qquad h(x)\cong |x|^{2\a} \quad \mbox{near } 0.
$$
Assume now by contradiction that $\wtilde u_1\not\equiv \wtilde u_2$ and suppose without loss of generality that $\wtilde u_2>\wtilde u_1$ in $\wtilde\O \subseteq\O$. Recall that $A+A'=B+B':=M$. Therefore, by \eqref{eq:toda-sing''} we have 
$$
	\D(\wtilde u_2-\wtilde u_1) +Mh(x)\bigr( e^{\wtilde u_2}-e^{\wtilde u_1} \bigr)=0.
$$
Note also that $h(x)\leq 1$. Since $\wtilde u_2>\wtilde u_1$ in $\wtilde\O$ we deduce
$$
	\D(\wtilde u_2-\wtilde u_1) +M\bigr( e^{\wtilde u_2}-e^{\wtilde u_1} \bigr)\geq 0 \qquad \mbox{in } \wtilde\O.
$$
With an argument similar to the one in the  proof of Theorem \ref{thm4} we get a contradiction. Thus $\wtilde u_1 \equiv \wtilde u_2 := \wtilde u$ and $\wtilde u$ satisfies  
$$
	\left\{ \begin{array}{rll}
-\D \wtilde u = & D h(x) e^{\wtilde u}  & \mbox{in } \O, \vspace{0.2cm}\\
 \wtilde u = & g(x) & \mbox{on } \p \O,
\end{array}
\right.
$$
where $D:= A-B=A'-B'$. Arguing as in the proof of Theorem \ref{thm4} we deduce that $\wtilde u$ is unique. 
\end{proof}

\

\begin{center}
\textbf{Acknowledgements}
\end{center}

The authors would like to thank Prof. G. Tarantello and Dr. W. Yang for the discussions concerning the topic of this paper. 

\

\end{document}